\documentclass[12pt,a4paper,reqno]{article}
\usepackage{amssymb,amsfonts,amsmath,amsthm}
\usepackage{url,hyperref}
\usepackage{fullpage}
\usepackage{enumitem}
\usepackage{tikz}
\usetikzlibrary{calc}

\theoremstyle{definition}
\newtheorem{theorem}{Theorem}[section]

\newtheorem{cor}[theorem]{Corollary}

\newtheorem{lemma}[theorem]{Lemma}
\newtheorem{prop}[theorem]{Proposition}
\newtheorem{defn}[theorem]{Definition}

\newtheorem{conjecture}[theorem]{Conjecture}

\newtheorem{example}[theorem]{Example}
\newtheorem{remark}[theorem]{Remark}
\numberwithin{equation}{section}

\newcommand{\N}{\mathbb{N}}
\newcommand{\sr}{\mathsf{sr}}

\newcommand{\ab}[1]{\langle #1 \rangle}

\newcommand{\cf}{\mathsf{cf}}
\newcommand{\CF}{\mathsf{CF}}
\newcommand{\Alist}{\mathsf{Alist}}
\newcommand{\Tchouk}{\mathsf{Tchouk}}

\newcommand{\final}{\mathrm{final}}
\newcommand{\length}{\mathrm{length}}
\newcommand{\GETOUT}[1]{}
\newcommand{\seqnum}[1]{\href{https://oeis.org/#1}{\rm \underline{#1}}}

\newcommand\tchoukb{ 
\begin{tikzpicture}
	\draw[fill=gray!10] (0,0) rectangle (7.6,1.6);
	\draw[fill=white] (0.4,0.4) rectangle (1.2,1.2);
	\node at (0.8,0.2) {\tiny{pit}};
	\foreach \i in {1,...,5}
		{
		\node at (\i*1.2+0.8,0.2) {\tiny{hole \i}};
		\draw[fill=white] (0.8+\i*1.2,0.8) circle (0.4);
		}
	\fill (2*1.2 + 0.8, 0.8 ) circle(2.5pt);
	\fill (3*1.2 + 0.8, 0.8 ) circle(2.5pt);
	\fill (4*1.2 + 0.95, 0.8 ) circle(2.5pt);
	\fill (4*1.2 + 0.75, 0.9 ) circle(2.5pt);
	\fill (4*1.2 + 0.75, 0.7 ) circle(2.5pt);
	\fill (5*1.2 + 1.0, 0.8 ) circle(2.5pt);
	\fill (5*1.2 + 0.84, 0.96 ) circle(2.5pt);
	\fill (5*1.2 + 0.84, 0.64 ) circle(2.5pt);
	\fill (5*1.2 + 0.64, 0.9 ) circle(2.5pt);
	\fill (5*1.2 + 0.64, 0.7 ) circle(2.5pt);
	\node at (3.8,1.9) {Configuration (0,1,1,3,5)};
\end{tikzpicture}
}
\newcommand\tchoukc{ 
\begin{tikzpicture}
	\draw[fill=gray!10] (0,0) rectangle (7.6,1.6);
	\draw[fill=white] (0.4,0.4) rectangle (1.2,1.2);
	\node at (0.8,0.2) {\tiny{pit}};
	\foreach \i in {1,...,5}
		{
		\node at (\i*1.2+0.8,0.2) {\tiny{hole \i}};
		\draw[fill=white] (0.8+\i*1.2,0.8) circle (0.4);
		}
	\fill (0*1.2 + 0.8, 0.8 ) circle(2.5pt);
	\fill (1*1.2 + 0.8, 0.8 ) circle(2.5pt);
	\fill (2*1.2 + 0.8, 0.91 ) circle(2.5pt);
	\fill (2*1.2 + 0.8, 0.69 ) circle(2.5pt);
	\fill (3*1.2 + 0.8, 0.91 ) circle(2.5pt);
	\fill (3*1.2 + 0.8, 0.69 ) circle(2.5pt);
	\fill (4*1.2 + 0.69, 0.69 ) circle(2.5pt);
	\fill (4*1.2 + 0.69, 0.91 ) circle(2.5pt);
	\fill (4*1.2 + 0.91, 0.69 ) circle(2.5pt);
	\fill (4*1.2 + 0.91, 0.91 ) circle(2.5pt);
	\node at (3.8,1.9) {Configuration (1,2,2,4,0)};
\end{tikzpicture}
}
\newcommand\tarrow{
	\begin{tikzpicture}
	\draw[thick,->] (0,0) -- (0,-0.5);
	\end{tikzpicture}
}

\begin{document}
\title{Fagan's Construction, Strange Roots, and Tchoukaillon Solitaire}
\author{Mark Dukes \\
School of Mathematics and Statistics\\ University College Dublin\\ Dublin 4, Ireland\\
\texttt{mark.dukes@ucd.ie}
}
\date{}

\maketitle
\begin{abstract}
In this paper we examine a procedure that, on starting with an integer $n$, results in a pair of equal integers that are no greater than $n$.
We call the resulting value the {\it{strange root}} of $n$ and we show how this strange-root-finding procedure is intimately linked to the game of Tchoukaillon solitaire. 
We analyze the strange-root-finding procedure in reverse to determine when a prescribed value is the strange root of at most two integers. 
We present a conjecture about strange roots and translate this conjecture into one involving Tchoukaillon solitaire.
\end{abstract}

\section{Introduction}
In this paper we will present some results relating an algorithmic procedure on integer pairs due to Colm Fagan~\cite{cf}, an actuary with a keen interest in mathematics,
to the game of Tchoukaillon solitaire. 
First we will explain Fagan's construction as it was originally defined in 
the On-line Encyclopedia of Integer Sequences~\cite[\seqnum{A204539}]{oeis}
and state the question that has motivated the work on this construction.

\medskip
\noindent 
{\bf{Fagan's Construction:}} 
{\em
Let $m$ be a positive integer. 
Define the first {\em{Fagan pair}} to be $(2,2m)$. 
If the current Fagan pair is $(i,y)$ and $y>i$, 
then construct the next Fagan pair $(i+1,z)$ where $z$ is the smallest integer such that $(i+1)z>iy$ and $i+1+z$ is even.
Do this until the current Fagan pair $(i,y)$ satisfies $y\leq i$.
}
\medskip

To illustrate this choose $m=4$. We begin with the Fagan pair $(2,8)$ which produces the next Fagan pair $(3,7)$. 
Applying the rule once again, we construct the Fagan pair $(4,6)$ followed by the Fagan pair $(5,5)$. 
As $5 \not > 5$ we are done.
Let us use $\CF(m)$ to refer to the resulting sequence of Fagan pairs in this case, i.e.,
$$\CF(4): ~~ (2,8) \to (3,7) \to (4,6) \to (5,5).$$

The outcome of this procedure seems to yield a pair of equal positive integers that we denote $(\cf(m),\cf(m))$, and in the above example $\cf(4)=5$. 
We prove this equality by using Lemma~\ref{samealistend} in conjunction with the connection to Fagan pairs that follows Lemma~\ref{tues}.

\medskip
\noindent {\bf{Fagan's Question:}} 
{\em
If there is only one integer $m$ for which $\cf(m) = n$, we say that $n$ is $\cf$-unique.
Are there infinitely many positive integers $n$ that are $\cf$-unique?
}

\medskip
The values of $m$ for which this is known to be true are 1, 2, 3, 6, 30, 493080, and 242650650.
The corresponding values of $n$ are given in the table and is sequence 
\seqnum{A204540}.
Information on the number of integers $m$ for which $\cf(m)$ is some prescribed value is given in 
\seqnum{A204539}.

$$
\begin{array}{|c|c|} \hline
m & \cf(m) \\ \hline
1 & 2 \\
2 & 3 \\
3 & 4 \\
6 & 6 \\
30 & 14  \\
493080 & 1760 \\
242650650 & 39046 \\ \hline
\end{array}
$$

Fagan's construction relies on the chosen integer being a multiple of 4 and contains a parity condition.
Observations and calculations made regarding this construction may be found in sequences 
\seqnum{A204539}, \seqnum{A204540}, and \seqnum{A185001}.
The original terminology includes the notions of {\it{basins}} and {\it{the sea}} which represent the size of the pre-image $\cf^{-1}(n)$ and the Fagan pair $(\cf(m),\cf(m))$, respectively.

In order to preserve Fagan's original construction but at the same time move to a more mathematically convenient framework, 
in this paper we will define (in Definition~\ref{alistdef}) a procedure on a larger set of integers and introduce the notion of a {\it{strange root of an integer $n$}}. 
This new procedure incorporates Fagan's construction.
The correspondence between strange roots and the final pair of numbers in a Fagan pair is that the strange root of $2m$ is equal to $\cf(m)$.

Tchoukaillon solitaire is a game played on a one-dimensional board that consists of a pit followed by a sequence of holes (see Figure~\ref{tchoukfigone}).
Every hole can contain a number of stones. The pit is special in that stones can be placed into it but not removed from it, and it is initially empty.
The aim of the game to move stones from all holes to the pit according to a redistribution rule: 
select a hole, pick up all the stones in that hole, and place one stone in each hole that one meets on the way to the pit. 
It turns out that for a fixed number $n$ of stones, there is only one configuration on a board consisting of $n$ stones that is winnable, 
and the order in which hole selection should be executed is non-trivial.
The game is won if one can select the holes in such an order so that at end all stones are in the pit.
Figure~\ref{tchoukfigone} illustrates the Tchoukaillon board for $n=10$ and the winning play of the game.

\begin{figure}
\begin{center}
\tchoukb\\[0.3em]
\tarrow\\
\tchoukc\\
$$\begin{array}{l@{\,}l@{\,}l@{\,}l@{\,}l}
(0,1,1,3,5) &\mapsto (1,2,2,4,0) & \mapsto (0,2,2,4,0) & \mapsto (1,0,2,4,0) & \mapsto (0,0,2,4,0) \\ 
			& \mapsto (1,1,3,0,0) & \mapsto (0,1,3,0,0) & \mapsto (1,2,0,0,0) & \mapsto (0,2,0,0,0) \\
			& \mapsto (1,0,0,0,0) & \mapsto (0,0,0,0,0)
\end{array}$$
\end{center}
\caption{The Tchoukaillon board configurations for $n=10$ is displayed first followed by the board configuration after the first play. Note that we do not record the number of stones in the pit in the configuration.
The list of board configurations resulting in a win is given.
\label{tchoukfigone}}
\end{figure}

It turns out that there are precisely two Tchoukaillon boards for which the rightmost non-empty hole is hole number 5, and these boards correspond to having $n=10$ and $n=11$ stones.
In this paper we will prove the surprising result, in Proposition~\ref{strangetchouk}, that the number of Tchoukaillon boards 
for which the rightmost non-empty hole is hole number $k-1$ is equal to the number of integers $n$ whose strange root is $k$.
The number of integers $n$ whose strange root is $5+1$ is precisely two (these are the integers 11 and 12).

We also address the question of whether it is possible to go backwards from a strange root and systematically derive those integers that will have a specified strange root.
It will transpire in Proposition~\ref{srinverse} that when we try working backwards from a strange root pair, there can be at most two pairs that map to a specified pair, and it can only be two if a divisibility condition is satisfied.
We end the paper by using the correspondences we have proven to give an equivalent formulation of Fagan's conjecture in terms of Tchoukaillon solitaire.

\section{Strange roots}
In this section we will consider a construction on the natural numbers that at first looks fundamentally different to Fagan's construction.
We will show that it incorporates Fagan's construction by way of a linear transformation.
As with Fagan's construction, our construction also appears to terminate in an equal pair of integers and we prove this equality in Lemma~\ref{samealistend}.
We will refer to the equal pair of values that each construction terminates as the {\it{strange root}} of the initial number.
To clearly distinguish our construction from Fagan's, we will use angle brackets in place of the parentheses used for Fagan pairs in Section 1.

Let $\N=\{1,2,\ldots\}$ be the set of natural numbers.

\begin{defn}\label{alistdef}
Let $n \in \N$.
Let $\Alist_n$ be the sequence of pairs produced by the following algorithm:
Begin with the pair $\ab{1,n}$. 
Given a pair $\ab{i,y_i}$ with $y_i>i$, construct the subsequent pair $\ab{i+1,y_{i+1}}$ where $y_{i+1}$ is the smallest integer such that $(i+1) y_{i+1}> i(y_i +1)$. 
Equivalently, $y_{i+1}$ is the unique integer such that $$ y_{i+1} > \dfrac{i(y_i +1)}{i+1} \geq y_{i+1} -1 .$$
This produces a sequence
$\Alist_n ~ = ~ \ab{1,n=y_1} \to \ab{2,y_2} \to \ab{3,y_3} \to \cdots \to \ab{\sr(n),y_{\sr(n)}}$
where $\ab{\sr(n),y_{\sr(n)}}$ is the final pair in this sequence.
We will find it convenient to call the value $\sr(n)$ the {\it{strange root}} of $n$.
\end{defn}

\begin{example}\begin{enumerate} \item[]
\item[(i)]
Suppose $n=2$. We begin with $\ab{1,2}$. 
As $y_1=2>1$ we let $y_2$ be the smallest integer greater than $1(2)/2 = 1$, which is 2. This gives the pair $\ab{2,2}$.
Since $y_2 \not > 2 $ we are done and $\ab{2,2}=\ab{\sr(2),y_{\sr(2)}}$. 
Thus $\Alist_2$ is $\ab{1,2} \to \ab{2,2}$.
\item[(ii)]
Suppose $n=8$. We start with $\ab{1,8}$.
As $8>1$ we let $y_2$ be the smallest integer greater than $1(8+1)/2=4.5$, which is 5.
We now have the pair $\ab{2,5}$ and since $5>2$ we let $y_3$ be the smallest integer greater than $2(5+1)/3 = 4$, which is 5.
This gives the pair $\ab{3,5}$.
As $5>3$ we let $y_4$ be the smallest integer greater than $3(5+1)/4 = 4.5$, which is 5.
This gives the pair $\ab{4,5}$.
As $5>4$ we let $y_5$ be the smallest integer greater than $4(5+1)/5 = 4.8$, which is 5.
This gives the pair $\ab{5,5}$.
Since $5\not > 5$ this is the final pair and so $\ab{\sr(8),y_{\sr(8)}}=\ab{5,5}$.
Thus $\Alist_8$ is $\ab{1,8} \to \ab{2,5}  \to \ab{3,5} \to \ab{4,5}  \to \ab{5,5}$.
\item[(iii)] The $\Alist$ sequences for the first few integers are illustrated in Figure~\ref{firstalist}.
\end{enumerate}
\end{example}

\begin{figure}
\begin{center}
$$
\begin{array}{|c|l|} \hline
n & \multicolumn{1}{c|}{\Alist_n}\\ \hline
1 & \ab{1,1} \\
2 & \ab{1,2} \to \ab{2,2} \\
3 & \ab{1,3} \to \ab{2,3} \to \ab{3,3}\\
4 & \ab{1,4} \to \ab{2,3} \to \ab{3,3} \\
5 & \ab{1,5} \to \ab{2,4} \to \ab{3,4} \to \ab{4,4}\\
6 & \ab{1,6} \to \ab{2,4} \to \ab{3,4} \to \ab{4,4}\\
7 & \ab{1,7} \to \ab{2,5} \to \ab{3,5} \to \ab{4,5} \to \ab{5,5}\\
8 & \ab{1,8} \to \ab{2,5} \to \ab{3,5} \to \ab{4,5} \to \ab{5,5}\\
9 & \ab{1,9} \to \ab{2,6} \to \ab{3,5} \to \ab{4,5} \to \ab{5,5}\\
10 & \ab{1,10} \to \ab{2,6} \to \ab{3,5} \to \ab{4,5} \to \ab{5,5}\\
11 & \ab{1,11} \to \ab{2,7} \to \ab{3,6} \to \ab{4,6} \to \ab{5,6} \to \ab{6,6}\\
12 & \ab{1,12} \to \ab{2,7} \to \ab{3,6} \to \ab{4,6} \to \ab{5,6} \to \ab{6,6}\\ 
13 & \ab{1,13} \to \ab{2, 8} \to  \ab{3, 7} \to \ab{4,7} \to  \ab{5,7} \to \ab{6, 7} \to  \ab{7, 7} \\
14 & \ab{1, 14} \to \ab{2, 8} \to  \ab{3, 7} \to \ab{4,7} \to  \ab{5,7} \to \ab{6, 7} \to  \ab{7, 7} \\
15 & \ab{1, 15} \to \ab{2, 9} \to  \ab{3, 7} \to  \ab{4, 7} \to  \ab{5, 7} \to  \ab{6, 7} \to  \ab{7, 7}\\
16 & \ab{1, 16} \to \ab{2, 9} \to  \ab{3, 7} \to  \ab{4, 7} \to  \ab{5, 7} \to  \ab{6, 7} \to  \ab{7, 7}\\
17 & \ab{1, 17} \to \ab{2, 10} \to  \ab{3, 8} \to  \ab{4, 7} \to  \ab{5, 7} \to  \ab{6, 7} \to  \ab{7, 7}\\
18 & \ab{1, 18} \to \ab{2, 10} \to  \ab{3, 8} \to  \ab{4, 7} \to  \ab{5, 7} \to  \ab{6, 7} \to  \ab{7, 7}\\
19 & \ab{1, 19} \to \ab{2, 11} \to  \ab{3, 9} \to  \ab{4, 8} \to  \ab{5, 8} \to  \ab{6, 8} \to  \ab{7, 8} \to  \ab{8, 8}\\
20 & \ab{1, 20} \to \ab{2, 11} \to  \ab{3, 9} \to  \ab{4, 8} \to  \ab{5, 8} \to  \ab{6, 8} \to  \ab{7, 8} \to  \ab{8, 8}\\ \hline
\end{array}
$$
\end{center}
\caption{
The first few $\Alist$ sequences as defined in Definition~\ref{alistdef}
\label{firstalist}
}
\end{figure}

\begin{lemma} \label{samealistend}
For every $n \in \N$, we have $\sr(n)=y_{\sr(n)}$.
\end{lemma}

\begin{proof}
Consider the sequence of pairs produced by Definition~\ref{alistdef}. 
If $\ab{\sr(n),y_{\sr(n)}}$ is the final entry of $\Alist_n$, then it must be the case that $y_{\sr(n)} \leq \sr(n)$.
We will now show the following:
\begin{enumerate}
\item[(i)]
If we have pairs $\ab{i,y_i}$ and $\ab{i+1,y_{i+1}}$ as part of this process, 
then $y_{i+1} \leq y_i$.
\item[(ii)]
If we have pairs $\ab{i,y_i}$ and $\ab{i+1,y_{i+1}}$ as part of this process, then $y_{i+1} \geq i+1$.
\end{enumerate}
For (i), since $\ab{i,y_i} \to \ab{i+1,y_{i+1}}$, by assumption it must be the case that $y_i>i$, so $y_{i}\geq i+1$.
By Definition~\ref{alistdef}, $y_{i+1}$ is the smallest integer that is strictly greater than the quantity $z=i(y_i+1)/(i+1) = y_i+1 - (y_i+1)/(i+1)$.
As $y_i \geq i+1$ we must have that $(y_i+1)/(i+1) \geq (i+2)/(i+1) = 1+ 1/(i+1)$.
Therefore we have $z = y_i+1 - (y_i+1)/(i+1) \leq y_i+1 - (1+ 1/(i+1)) = y_i- 1/(i+1)$.
As $y_{i+1}$ is the smallest integer strictly greater than $z$, and since $z \leq y_i- \frac{1}{i+1}<y_i$, we must therefore have $y_{i+1} \leq y_i$.

For (ii), notice that since $z = y_i+1 - (y_i+1)/(i+1)$ and $y_i \geq i+1$, we have $z \geq (i+1) - 1/(i+1)$.
As $y_{i+1}$ is the smallest integer strictly greater than $z$, we must have $y_{i+1} \leq i+1$.

The implications of (i) and (ii) are that the sequence of pairs must terminate since the second value is weakly decreasing and the first value is strictly increasing.
Part (ii), with $i=\sr(n)-1$ gives us the inequality $y_{\sr(n)} \geq \sr(n)$. 
Since $y_{\sr(n)} \leq \sr(n)$ we must have that there is a final pair, 
and this final pair is $\ab{\sr(n),y_{\sr(n)}=\sr(n)}$, as claimed.
\end{proof}

The strange roots for the first few integers are given in Figure~\ref{firststrangeroots}.
\begin{figure}[!h]
\begin{center}
\begin{tabular}{|r|c|c|c|c|c|c|c|c|c|c|c|c|c|c|c|c|c|c|c|c|} \hline
$n$ &      1 & 2 & 3 & 4 & 5 & 6 & 7 & 8 & 9 & 10 & 11 & 12 & 13 & 14 & 15 & 16 & 17 & 18 & 19 & 20 \\ \hline
$\sr(n)$ & 1 & 2 & 3 & 3 & 4 & 4 & 5 & 5 & 5 & 5  & 6  &  6 &  7 &  7 &  7 & 7  &  7 &  7 &  8 &  8 \\ \hline
\end{tabular}
\end{center}
\caption{
The strange roots of the first few integers.
\label{firststrangeroots}
}
\end{figure}
A first observation is that the sequence seems to be weakly increasing. 
\begin{lemma}\label{tues}
Let $n_1,n_2 \in \mathbb{N}$ with $n_1 \leq n_2$. Then $\sr(n_1) \leq \sr(n_2)$.
\end{lemma}

\begin{proof}
Suppose we have
\begin{align*}
\Alist_{n_1} &= \ab{1,n_1=x_1} \to \ab{2,x_2} \to \cdots \to \ab{\sr(n_1),\sr(n_1)},\\
\Alist_{n_2} &= \ab{1,n_2=y_1} \to \ab{2,y_2} \to \cdots \to \ab{\sr(n_2),\sr(n_2)}.
\end{align*}
Let us assume we have the pairs $\ab{i,x_i}$ and $\ab{i,y_i}$ and assume that $x_i \leq y_i$. Note that, by assumption this is true for $i=1$.
By Definition~\ref{alistdef} we have that $x_{i+1}$ is the smallest integer strictly greater than $i(x_i+1)/(i+1)$. 
Since $x_i\leq y_i$ we also have $i(x_i+1)/(i+1) \leq i(y_i+1)/(i+1)$, and by definition $y_{i+1}$ is the smallest integer strictly greater than $i(y_i+1)/(i+1)$. 
Therefore $x_{i+1} \leq y_{i+1}$.

This tells us that at the second entry in $\Alist_{n_2}$ is always at least as big as the corresponding second entry in $\Alist_{n_1}$, and the result follows.
\end{proof}

Let us now connect these pairs of integers we have just introduced to Fagan pairs.
Consider $\Alist_{2n}: \ab{1,2n} \to \cdots \to \ab{\sr_{2n},\sr_{2n}}$ and
the linear transform $\ab{x,y} \mapsto (x,2y-x)$. 
Note that the sum of the entries in the resulting pair is always even.
The inverse of this transform is $(x,y) \mapsto \ab{x,(x+y)/2}$.
Under this transform the first entry of our sequences $\ab{1,2n} \mapsto (1,4n-1)$. 
The second pair in $\Alist_{2n}$ is $\ab{2,y_2}$ where $y_2$ is the smallest integer such that 
$y_2>(y_1+1)/2 = (2n+1)/2$, which is $n+1$. The pair $\ab{2,n+1} \mapsto (2,2n)$ which is a Fagan pair.

Let us suppose that the $i$th pair in $\Alist_{2n}$ is $\ab{i,y_i}$, which maps to $(i,2y_i-i)$.
If $y_i>i$ then the subsequent pair in $\Alist_{2n}$ will be $\ab{i+1,y_{i+1}}$ which maps to $(i+1,2y_{i+1} - (i+1))$ where $y_{i+1}$ is the smallest integer such that $y_{i+1} > i(y_i +1)/(i+1)$. 

Let us re-frame this without the use of the angle bracket pairs.
Define $a_i := 2y_i-i$.
Notice that $y_i>i$ $\Leftrightarrow 2y_i>2i \Leftrightarrow a_i>i$.
Given the pair $(i,a_i)$ with $a_i>i$ and $i+a_i$ even, 
we next produce the pair $(i+1,a_{i+1})$ where $a_{i+1}$ is the smallest integer such that 
$$\dfrac{a_{i+1}+(i+1)}{2} > i \dfrac{\left( \left(\frac{a_i+i}{2} \right) + 1 \right)}{i+1}.$$
This inequality simplifies to 
$$(i+1)a_{i+1} \geq i a_i,$$
with the additional property that $i+1+a_{i+1}$ is even.

We would like to show that equality can never occur in this inequality if $i+a_i$ and $i+1+a_{i+1}$ are both even.
Suppose that $i$ is even. 
Then $i+1$ is odd and since $i+a_i$ is even we must have that $a_i$ is even. 
Since $i+1+a_{i+1}$ is even we must have that $a_{i+1}$ is odd.
The product $(i+1)a_{i+1}$ is therefore off and the product $ia_i$ is even, and so it is not
possible for $(i+1)a_{i+1}$ to be equal to $ia_i$.
A similar argument holds for when $i$ is odd.
With equality ruled out, we have that $a_{i+1}$ is the smallest integer such that $(i+1)a_{i+1} > ia_i$ and such that $i+1+a_{i+1}$ is even.
This statement is the defining step of Fagan's construction and so all of these pairs beginning with $(2,2m)$ are Fagan pairs.

To conclude, the sequence of Fagan pairs that one meets in Fagan's construction for the integer $4m$ corresponds exactly to the sequence of pairs in $\Alist_{2m}$ (discounting the initial pair) via the linear transformation $\ab{x,y} \mapsto (x,2y-x)$.
The final pair in $\Alist_{2m}$ is $\ab{\sr(2m),\sr(2m)} \mapsto (\sr(2m),2\sr(2m)-\sr(2m)) = (\sr(2m),\sr(2m)) = (\cf(m),\cf(m)).$

\begin{prop}\label{twosix}
Let $n \in \N$  and define $y_1=n$.
Then the sequence of numbers $(y_1,\ldots,y_r)$ are such that
$\Alist_n$ is $\ab{1,y_1} \to \ab{2,y_2} \to \cdots \to \ab{r,r=y_r}$ iff 
the sequence of numbers $w_i$ defined by
$w_i = y_i-y_{i+1}+1$ satisfies
$$w_i =  \left\lceil  \dfrac{n-1-(w_1+\cdots+w_{i-1})}{i+1}\right\rceil,$$
for all $i=1,2,\ldots,r-1$. (Note that the sum of $w_j$'s is zero when $i=1$.)
\end{prop}

\begin{proof}
Let $n \in \N$ and set $y_1:=n$.
By Definition~\ref{alistdef}, a sequence of numbers $(y_1,\ldots,y_{r})$ is such that
$\ab{1,y_1} \to \ab{2,y_2} \to \cdots \to \ab{r,y_{r}}$ with $r=\sr(y_1)=y_r$ 
iff  $y_{i+1} > {i(y_i+1)}/{(i+1)} \geq y_{i+1}-1$ for all 
$i \in \{1,\ldots,r-1\}$ and $y_r=r=\sr(y_1)$. 

Let us translate this last statement into one that concerns only the differences between the $y_i$'s.
For a general sequence of numbers $(y_1,\ldots,y_r)$, we will consider the sequence $(z_1,\ldots,z_{r-1})$ of differences
where $z_i := y_i-y_{i+1}$ for all $1\leq i < r$.

The value $z_i$ comes from the transition $\ab{i,y_i} \to \ab{i+1,y_{i+1}}$.
Since there is a transition we must have $y_i>i$ and, by Definition~\ref{alistdef},
$y_{i+1}$ is the unique integer such that
$$ y_{i+1} > \dfrac{i(y_i+1)}{i+1} \geq y_{i+1}-1 .$$
Subtracting every term in this inequality from $y_i$ gives
$$y_i-y_{i+1} < y_i - \dfrac{i(y_i+1)}{i+1} = \dfrac{y_i+1}{i+1}-1 \leq y_i-y_{i+1}+1,$$
which is equivalent to
$$z_i < \dfrac{y_i+1}{i+1}-1 \leq z_{i} +1.$$
This, in turn, is equivalent to $\left\lceil \dfrac{y_i+1}{i+1}-1 \right\rceil = z_i +1$, i.e.,
$\left\lceil \dfrac{y_i+1}{i+1}\right\rceil = z_i +2$.
We can now provide an expression for the $z_i$ values without the $y_i$ values by noticing that the sum of the first $i-1$ $z$ values
is $z_1+\cdots+z_{i-1} = y_1-y_i = n-y_i$. This gives
$$z_i  = \left\lceil \dfrac{1+n-(z_1+z_2+\cdots+z_{i-1})}{i+1}\right\rceil -2$$
for all $i=1,\ldots,r-1$. Note that when $i=1$ the sum of the $z$ terms in the expression is empty, and is consequently 0.

Substituting $z_i = w_i -1$ into the above expression, and simplifying, gives
$$w_1 =  \left\lceil  \dfrac{n-1}{2} \right\rceil$$
and for all $i=2,\ldots,r-1$,
$$w_i =  \left\lceil  \dfrac{n-1-(w_1+\cdots+w_{i-1})}{i+1}\right\rceil.$$
\end{proof}

\section{Tchoukaillon solitaire}
Let us now introduce the board game Tchoukaillon solitaire and detail some of its properties.
The board for this game is a sequence of holes numbered $0,1,2,\ldots$. 
We will assume that hole 1 is to the right of hole 0, and hole 2 is to the right of hole 1, and so on (see Figure~\ref{tchoukfigone}).
The game is played as follows: $n$ stones are placed in these holes, but hole 0 (also known as the pit) is special and does not initially contain any stones.

The aim of the game is to move all the stones in holes 1 and above to the pit through some sequence of valid moves.
A valid move consists of selecting a hole, $i$ say, that currently contains $s_i$ stones, and then re-distributing these $s_i$ stones by placing one stone 
into each of the $s_i$ holes $i-1$, $i-2$, $\ldots$, $i-s_i$. 
If $i-s_i<0$ then we have no holes left in which to place the remaining $s_i-i$ stones, and we immediately {\it{lose}} the game.
One should therefore never select a hole that currently has more stones than there are holes to its left. 
The game is won if one can select the holes in such an order that we end up with all stones in the pit.

Let us write $c~=~ (c_1,c_2,\ldots )$ for a Tchoukaillon configuration whereby $c_i$ is the number of stones in hole $i$ and $n:=c_1+c_2+\cdots$.
It turns out that for every $n$ there is a unique winning Tchoukaillon configuration, $\Tchouk_n$, consisting of $n$ stones.
We list these configurations in Figure~\ref{tconfigs}.

We find it important to mention that it is {\it{not}} the case that any order of selecting holes in $\Tchouk_n$ results in a win. 

\begin{example}\label{notany}
Consider $\Tchouk_3=(1,2)$. If we select hole 1 first, then on performing our move the single 
stone is placed into the pit. 
Next we select hole 2 that contains two stones, and on performing our move we drop one stone into hole 1 and the other into the pit. 
Next we select hole 1 again, and drop the stone from there into the pit. 
After this all three stones are in the pit and we have won the game.

However, had we selected hole 2 first, then one stone would have been placed into hole 1, and one into the pit.  
There are then two stones in hole 1, and there is no way of winning so we will have lost the game.
\end{example}

\begin{figure}[!h]
$$
\begin{array}{|@{\qquad}r@{\qquad}|c@{\qquad}l@{\qquad}c|c@{\qquad}l@{\qquad}c|} \hline
n & \multicolumn{3}{c|}{\mbox{$\Tchouk_n$}} & \multicolumn{3}{c|}{\mbox{Move vector}} \\ \hline
0 &&  ()&&& none &\\
1 &&  (1)&&& (1) &\\
2 &&  (0,2)&&& (1,1) &\\
3 &&  (1,2)&&& (2,1) & \\
4 &&  (0,1,3)&&& (2,1,1) &\\
5 &&  (1,1,3)&&& (3,1,1) &\\
6 &&  (0,0,2,4)&&& (3,1,1,1) &\\
7 &&  (1,0,2,4)&&& (4,1,1,1) & \\
8 &&  (0,2,2,4)&&& (4,2,1,1) &\\
9 &&  (1,2,2,4)&&& (5,2,1,1)& \\
10 && (0,1,1,3,5)&&& (5,2,1,1,1)& \\
11 && (1,1,1,3,5)&&& (6,2,1,1,1)& \\
12 && (0,0,0,2,4,6)&&& (6,2,2,1,1)& \\
13 && (1,0,0,2,4,6)&&& (7,2,1,1,1,1) & \\ \hline
\end{array}
$$
\caption{The first few unique winning configurations $\Tchouk_n$ of Tchoukaillon solitaire
\label{tconfigs}}
\end{figure}

One way to construct $\Tchouk_n$ is by recursion. 
Given $\Tchouk_{n-1}$, suppose that position $i$ is the leftmost position containing 0 stones.
Then $\Tchouk_{n}$ is the configuration that results from $\Tchouk_{n-1}$ by adding $i$ stones to hole $i$, 
and subsequently removing one stone from each of the holes $1,2,\ldots,i-1$.

Recently, Jones, Taalman, and Tongen~\cite{jtt} gave an explicit method to construct the winning configurations $\Tchouk_n$.
The configuration $\Tchouk_n = (c_1,c_2,\ldots)$ whereby 
\begin{align*}
c_1 & = n \bmod 2\\
c_2 &= n-c_1 \bmod 3 \\
c_3 &= n-(c_1+c_2) \bmod 4 \\
& \;\;\vdots \\
c_{k} &= n - (c_1+c_2+\cdots+c_{k-1}) \bmod (k+1).
\end{align*}
Once the sum $c_1+c_2+\cdots+c_{k-1} = n$ one stops computing further entries.

For any winning configuration, there is some sequence of moves that will result in a win. 
As we saw in Example~\ref{notany}, it is not the case that any sequence of moves on $\Tchouk_n$ will result in a win.
The sequence of moves that is required to `win' can be discovered by playing the game in reverse, and is essentially the same as the recursive rule for constructing Tchoukaillon configurations highlighted above (except, of course, executed in reverse since it starts from the empty board).

There is another interesting aspect to the winning configurations that instead looks at the {\it number} of times each hole was selected 
for a valid move during a `win'. 
Given $c=(c_1,\ldots,c_k)=\Tchouk_n$, 
let $m=m(c)=(m_1,\ldots,m_k)$ be the sequence whereby $m_i$ is the number of times that hole $i$ was selected in playing the game.
This sequence $m$ is known in the literature as the {\it{move vector}}.
For example, in Example~\ref{notany} we considered $c=(1,2)=\Tchouk_3$. 
For that game hole 1 was selected twice (so $m_1=2$) and hole 1 was selected once (so $m_2=1$). 
The move vector for this $c$ is $m(c) = (2,1)$.

Taalman et al.~\cite{manm} gave an explicit expression for the entries of the move vector of $\Tchouk_n$ in terms of $n$:

\begin{theorem} (Taalman et al.~\cite[Thm.4]{manm})
The move vector for solving $\Tchouk_n$ is 
$m=(m_1,\ldots,m_{\ell})$ where
\begin{align*}
m_1 &= \left\lceil \dfrac{n}{2} \right\rceil \\
m_2 &= \left\lceil \dfrac{n-m_1}{3} \right\rceil \\
m_3 &= \left\lceil \dfrac{n-(m_1+m_2)}{4} \right\rceil \\
& \;\;\vdots \\
m_{\ell} &= \left\lceil \dfrac{n-(m_1+m_2+\cdots+m_{\ell - 1})}{\ell+1} \right\rceil.
\end{align*}
\end{theorem}

The link between Tchoukaillon solitaire and strange roots is now seen by comparing the theorem above to the expression in Proposition~\ref{twosix}.
The precise correspondence is given in the following theorem.

\begin{theorem} \label{biject}
Let $n \geq 2$.
The sequence $\Tchouk_{n-1}=(b_1,\ldots,b_{\ell})$ corresponds uniquely to 
$\Alist_n: ~ \ab{1,y_1=n} \to \ab{2,y_2}\to\cdots \to \ab{\sr(n),\sr(n)}$ as follows: 
\begin{enumerate}
\item[(a)] $\ell = \sr(n)-1$.
\item[(b)] For $i=1,2,\ldots,\sr(n)-1$,  
	$$b_i = 2i+1+iy_i - (i+1)y_{i+1}.$$
\item[(c)] For $i=1,2,\ldots,\sr(n)$, $$y_i = i + \dfrac{1}{i} (b_i+b_{i-1}+\cdots + b_{\sr(n)-1}).$$
\end{enumerate}
\end{theorem}

\begin{proof}
With the objects as stated in the theorem, the correspondences are established through the intermediate object of the move sequence 
$m=(m_1,m_2,\ldots,m_{\ell})$ where $m_i := 1+y_i-y_{i+1}$. The largest value of $\ell$ for which this is well defined is $\ell =  \sr(n)-1$, hence (a).
That $m$ is a valid move vector is verified by showing $0\leq m_i <i$ for all $i$, and this is a consequence of Lemma~\ref{samealistend}.

The sum of the entries in a move sequence is the same as the number of stones in the Tchoukaillon game, and so
$\sum_i m_i = (1+y_1-y_2)+(1+y_2-y_3)+\cdots = (\sr(n)-1) + y_1-y_{\sr(n)} = \sr(n) -1 + n - \sr(n) = n-1$.
In other words the sequence $\Alist_n$ can be seen to correspond to a move sequence for a Tchoukaillon game with $n-1$ stones.
In order to be able to write the entries of the sequences $(b_1,\ldots,b_{\sr(n)-1})$ and $(y_1,\ldots,y_{\sr(n)})$ in terms of one another, 
we will make use of some identities.
\begin{enumerate}
\item[(b)] To describe the $b_j$'s in terms of $m_j$'s, we use the following identify from Taalman et al.~\cite[Theorem 2]{manm}: $$b_i = im_i - \sum_{j=i+1}^{\sr(n)-1} m_j.$$
Substitute $m_j=1+y_j-y_{j+1}$ into this. We have
$$b_i = \begin{cases}
	i(1+y_i-y_{i+1}) -\sum_{j=i+1}^{\sr(n)-1} (1+y_j-y_{j+1}), & \text{ if $i\leq \sr(n)-2$;} \\
	(\sr(n)-1) (1+y_{\sr(n)-1} - y_{\sr(n)}), & \text{ if $i = \sr(n)-1$.}
	\end{cases}
$$
The expression in the top case simplifies to $2i+1-\sr(n) + iy_i - (i+1)y_{i+1} + y_{\sr(n)} = 2i+1+iy_i-(i+1)y_{i+1}$. 
The expression in the bottom case simplifies, by using the fact that $y_{\sr(n)-1} = y_{\sr(n)}$ for $n\geq 2$, to $\sr(n)-1$.
In fact if we use $i=\sr(n)-1$ in the top case then it reduces to this same expression $\sr(n)-1$, and so 
$b_i = 2i+1+iy_i-(i+1)y_{i+1}$ for all $i=1,2,\ldots,\sr(n)-1$.
\item[(c)] To describe the $y_j$'s in terms of the $b_j$'s.
If $\Tchouk_{n-1} = (b_1,b_2,\ldots,b_{\sr(n)-1}) = b$ then 
the move vector corresponding to $b$ is $m=(m_1,\ldots,m_{\sr(n)-1})$ where
\begin{align}
m_i & ~=~ \dfrac{1}{i} b_i + \dfrac{1}{i(i+1)} \sum_{k=i+1}^{\sr(n)-1} b_k , \label{useful}
\end{align}
for all $1\leq i \leq \sr(n)-1$ by using Taalman et al.~\cite[Eq.~(2)]{manm}.
As the $m_i$ and $y_i$ values are related via $m_i = 1+y_i-y_{i+1}$ for all $1\leq i \leq \sr(n)-1$ 
and $y_1=n$, we 
find that $$y_i = n+(i-1) - \sum_{t=1}^{i-1} m_t$$ and this holds for all $1\leq i \leq \sr(n)$.
The ends of this sequence are well defined since $y_1 = n+0-0=n$ and $y_{\sr(n)} = n+(\sr(n)-1) - (m_1+\cdots+m_{\sr(n)-1}) = n+(\sr(n)-1) - (n-1) = \sr(n)$.

Again by using Equation~(\ref{useful}) 
we can express the partial sum 
\begin{align*}
\sum_{t=1}^{i-1} m_t 
&= \sum_{t=1}^{i-1} \left(\dfrac{1}{t}b_t + \left(\dfrac{1}{t}-\dfrac{1}{t+1}\right) \sum_{k=t+1}^{\sr(n)-1} b_k\right) \\
&= \sum_{t=1}^{i-1} \dfrac{1}{t}b_t +   \sum_{t=1}^{i-1} \left(\dfrac{1}{t}-\dfrac{1}{t+1}\right) \sum_{k=t+1}^{\sr(n)-1} b_k \\
&= \sum_{t=1}^{i-1} \dfrac{1}{t}b_t +    \sum_{k=2}^{\sr(n)-1} b_k \sum_{t=1}^{\min(i,k)-1} \left(\dfrac{1}{t}-\dfrac{1}{t+1}\right) \\
&= \sum_{t=1}^{i-1} \dfrac{1}{t}b_t +    \sum_{k=2}^{\sr(n)-1} b_k \left( 1 - \dfrac{1}{\min(i,k)} \right) \\
&= \sum_{t=1}^{i-1} \dfrac{1}{t}b_t +   \left( \sum_{t=1}^{i-1} b_t\left(1-\frac{1}{t} \right)\right)    +   \left(1-\dfrac{1}{i} \right) (b_i +\cdots + b_{\sr(n)-1}) \\
&= (b_1+\cdots + b_{i-1})    +   \left(1-\dfrac{1}{i} \right) (b_i +\cdots + b_{\sr(n)-1}) \\
&= (n-1) - \dfrac{1}{i} \sum_{k=i}^{\sr(n)-1} b_i.
\end{align*}
Using this in the equation for $y_i$, and simplifying, we have that 
$$ y_i ~=~ i + \dfrac{1}{i} \sum_{k =i}^{\sr(n)-1} b_i,$$
for all $1\leq i \leq \sr(n)$.
Therefore the configuration $b=\Tchouk_{n-1}$ corresponds to 
$\Alist_n: \ab{1,y_1} \to \ab{2,y_2} \to \cdots \to \ab{\sr(n),\sr(n)}$, where
the $y_i$'s are as stated.
\end{enumerate}
\end{proof}

Notice that the end of a Tchoukaillon configuration is a fixed point in the following sense:

\begin{lemma}\label{littlelemma}
Suppose $b=(b_1,\ldots,b_{\ell}) = \Tchouk_n$. Then $b_{\ell}=\ell$.
\end{lemma}

\begin{proof}
This is straightforward to see by using the recursive construction presented after Example~\ref{notany}
Since $\Tchouk_1=(1)$, we have $b_{\ell}=1=\ell$ and it is true.
Suppose it is true for $n=k$ so that $\Tchouk_k = (b_1,\ldots,b_{\ell})$ with $\ell=b_{\ell}$.
To construct $\Tchouk_{k+1}$ from $\Tchouk_k$ we must condition on the appearance of the first (i.e., lowest indexed) 0 in $\Tchouk_k$.
\begin{enumerate}
\item[(a)] If $b_i=0$ is the first zero of $\Tchouk_k$ and $i<\ell $ then only the entries in holes $\{1,\ldots,i\}$ are changed and the final entry of $\Tchouk_{k+1}$ will be the same as the final entry of $\Tchouk_k$, hence $b_{\ell}=\ell$.
\item[(b)] If $b_{\ell+1}$ is the first zero of $\Tchouk_k$, then $\Tchouk_{k+1}$ must have $b_{\ell+1} = \ell +1$ and all entries to the left of this are decreased by one. 
\end{enumerate}
In both cases, the claim holds true and the result follows by induction.
\end{proof}

A natural corollary of Theorem~\ref{biject} and Lemma~\ref{littlelemma} is the following, 
which allows us to interpret questions about the strange root of $n$ in terms of winning Tchoukaillon configurations.

\begin{cor} For all $n \geq 2$,
$\sr(n) = \length(\Tchouk_{n-1}) = \final(\Tchouk_{n-1})$, where $\length(c)$ is the highest index $i$ such that $c_i\neq 0$, and $\final(c)$ is the value of that $c_i$.
\end{cor}

The correspondence established in this section allows us to gain some insight into the $\sr$ statistic through enumerative results on Tchoukaillon solitaire. 
The quantity that is most well-known in relation to Tchoukaillon solitaire is a statistic $t(k)$ that is defined as the smallest integer $n$ for which $k$ occurs for the first time in $\Tchouk_n$. 
For example, if we look at Figure~\ref{tconfigs}, we see that 4 first occurs in $\Tchouk_6$, and so $t(4)=6$.

Since the end of every winning Tchoukaillon sequence is a value equal to its index (by Lemma 3.4), $t(k)$ may be equivalently defined as the number of $n (\geq 0)$ for which there are no entries in holes $k,k+1,k+2,\ldots$. For example, for $k=4$, if we look at Figure~\ref{tconfigs} then there are no stones in holes 4 or higher of the configurations $\Tchouk_0$, $\Tchouk_1$, $\ldots$, $\Tchouk_5$ and so $t(4)=6$.

The sequence $(t(1),t(2),\ldots)$ is 
\seqnum{A002491}
and begins $$1,2,4,6,10, 12, 18, 22, 30, 34, 42, 48, 58, 60,\ldots.$$
It is known to have several curious properties. An extremely good exposition of these properties and further references can be found in the Jones et al.\ paper~\cite{jtt}.
\begin{itemize}
\item $t(k)$ can be calculated by starting with $k$  and successively rounding up to next multiple of $k-1$, $k-2$, $\ldots$, 1.
For example, $t(4)$ is calculated by starting with 4, round up to the next multiple of $k-1=3$ which is 6. Round up again to the next multiple of $k-2=2$ which is still 6, and rounding up to the next multiple of $k-3=1$ will not change the value at all. Thus $t(4)=6$. (Brown~\cite{brown}).
\item It can be generated by a sieving process on the integers. This was described in Erd\H{o}s \& Jabotinsky~\cite{erdos} and David~\cite{david}, and is very clearly explained in Sloane~\cite{sloanemfis}.
\item $t(k) = \dfrac{k^2}{\pi} + O(n)$ (this result is due to Broline \& Loeb~\cite{bl} and improves on Erd\H{o}s  \& Jabotinsky~\cite{erdos} result $t(k) = \dfrac{k^2}{\pi} + O(n^{4/3})$).
\end{itemize}

Brown's construction (in the first point above) bears a similarity to the construction that we are considering.
It produces pairs of integers according to a rule similar to ours. 
However, it does not stop in the same manner that Fagan's construction or Definition~\ref{alistdef} do, and so the notion of a `root' seems to have been skipped over.
In light of the correspondences we have established, we have the following:

\begin{prop}\label{strangetchouk}
\begin{enumerate}
\item[]
\item[(a)]
The number of integers $n$ for which $\sr(n)=k$ is $t(k)-t(k-1)$.
Equivalently, 
$t(k) = |\{n\geq 1~:~ \sr(n)\leq k\}|$.
\item[(b)] 
The number of integers whose strange root is less than $k$ is approximately $k^2/\pi$ for $k$ large. 
\end{enumerate}
\end{prop}
Part (b) helps justify the {\it strange root} terminology we have used as the number of non-negative integers whose {\it square root} is less than a natural number $k$ is $k^2$.
Although these connections give us some interesting information about the $\sr$ function, 
the known properties of $t(k)$ are not sufficient to aid us any further in considering Fagan's question. 
In the next section we will present a brief analysis of the $\Alist$ sequences with Fagan's question in mind.

\section{Determining integers having a prescribed strange root}
When we consider the sequence of pairs that arise from these constructions, is it possible to express those pairs that must precede some pair in a given sequence?
Moreover, given an integer $r$, is it possible to determine the set $\{n\in \N ~:~ \sr(n)=r\}$?

Suppose $$\Alist_n: ~ \ab{1,n} \to \cdots \to \ab{i,u} \to \ab{i+1,v} \to\cdots \to \ab{r,r}$$ where $r\geq 2$.
Definition~\ref{alistdef} tells us that $u$ must be an integer that satisfies $(i+1)v>i(u+1) \geq (i+1)(v-1)$.
In analyzing the values that $u$ can take, at a second glance it is more restricted than first appears. 
It transpires that there can be either one or two values of $u$ that map to a given $\ab{i+1,v}$.

\begin{prop}\label{srinverse}
Let $n \in \N$ and consider 
$$\Alist_n: ~ \ab{1,n} \to \cdots \to \ab{i,u} \to \ab{i+1,v} \to\cdots \to \ab{r,r}$$ where $r\geq 2$. 
Then 
$$u \in 
	\begin{cases}
	\left\{v-2+\left\lfloor \frac{v-1}{i} \right\rfloor, ~ v-1+\left\lfloor \frac{v-1}{i}\right\rfloor \right\}, & \text{ if $i ~ | ~ v-1$;}\\
	\left\{ v-1+\left\lfloor \frac{v-1}{i}\right\rfloor  \right\},
			& \text{ if $i ~   \not |  ~ v-1$.}
	\end{cases} $$
\end{prop}

\begin{proof}
Suppose that $ \ab{i,u} \to \ab{i+1,v}$ as stated in the theorem.
Then by Definition~\ref{alistdef}  $u$ must satisfy 
$(i+1)v>i(u+1) \geq (i+1)(v-1)$. 
This inequality is equivalent to $v+(v/i)-1=v-1+(v/i)>u \geq (v-1) + (v-1)/i -1=v-2 +(v-1)/i$, i.e.,
$$ v-1+ \left\lfloor \frac{v-1}{i} \right\rfloor \geq  u \geq v-2+\left\lceil \frac{v-1}{i} \right\rceil.$$
Notice that if $\frac{v-1}{i}$ is an integer $x$, then this inequality is 
$v-1+x \geq u \geq v-2+x$, i.e., $u \in \{v-2+x,v-1+x\}$.
However, if $\frac{v-1}{i}$ is not an integer but is $x+\epsilon$ for some integer $x \in \N$ and $\epsilon \in (0,1)$, then 
this inequality is 
$v-1+x \geq u \geq v-2+x+1 = v-1+x$, i.e., $u = v-1+x$.
For example, consider $\ab{i+1,v}=\ab{4,5}$, we have that $v-1=4$ and $i=3$. 
As $4/3$ is not an integer, the only $u$ for which $\ab{i,u} \to\ab{i+1,v}$ is $u=v-1+\lfloor (v-1)/x \rfloor = 4+1 = 5$. 
The only pair $\ab{3,u}$ that will produce $\ab{4,5}$ is $\ab{3,5}$.
\end{proof}

Proposition~\ref{srinverse} allows us to give a description of those integers $n$ whose strange root is some prescribed value by working backwards from the value of the root.
Let us observe that in Proposition \ref{srinverse}, when $i=1$, the value $i$ will always divide $v-1$, and there will be two possible values for $u$ such that $\ab{1,u} \to \ab{2,v}$ for all $v \geq 2$. Thus given a pair $\ab{2,v}$, both $\ab{1,2v-3} \to \ab{2,v}$ and $\ab{1,2v-2} \to \ab{2,v}$.
 
The following proposition provides a characterization of the $r$ that are the roots of at most two integers.

\begin{prop}\label{usr}
Suppose that $n\geq 5$. 
Let $x_{r} = r$ and for every $i=r-1,\ldots,1$ define 
$$x_{i} ~:=~ x_{i+1}-1 + \left\lfloor \frac{x_{i+1}-1}{i} \right\rfloor ~=~ \left\lfloor\frac{(i+1)(x_{i+1}-1)}{i} \right\rfloor.$$
Then there are only two integers ($x_1$ and $x_1-1$)
that have $r$ as its strange root 
if and only if 
$x_{i+1} -1 \not\equiv 0$ (mod $i$)
for all $i \in \{2,\ldots ,r-2\}$.
\end{prop}

\begin{example}
Consider $r=14$. 
Then we have $x_{14}= 14$, and we apply the rule to derive the second row of the following table:
\begin{center}
$$
\def\arraystretch{2.4}
\begin{array}{|c|cccccccccccccc|} \hline
i & 14 & 13 & 12 & 11 & 10 & 9 & 8 & 7 & 6 & 5 & 4 & 3 & 2 & 1 \\ \hline
x_{i} & 14 & 14 & 
14  
&
14 
& 
14 
& 
14 
& 
14 
& 
14 
& 
15 
& 
16 
& 
18 
& 
22 
& 
31 
& 
60 
\\  \hline
\dfrac{x_{i+1}-1}{i} &
& 
& 
\frac{13}{12}
& 
\frac{13}{11}
& 
\frac{13}{10}
& 
\frac{13}{9}
& 
\frac{13}{8}
& 
\frac{13}{7}
& 
\frac{13}{6}
& 
\frac{14}{5}
& 
\frac{15}{4}
& 
\frac{17}{3}
& 
\frac{21}{2}
& \\ \hline
\end{array}
$$
\end{center}
Using the top two rows we can compute the values in the bottom row.
None of the quotients in the bottom row are integers hence, by the above proposition, there are only two integers ($x_1=60$ and $x_1-1=59$) that have 14 its strange root.
\end{example}

Proposition~\ref{usr} classifies those $r$ that are the strange root of only two integers. 
There are precisely $r-3$ (non-)divisibility conditions to be satisfied in order for $r$ to be a unique strange root. 
Thus as $r$ grows it would appear less and less likely to find an $r$ such that the sequence $(x_r,\ldots,x_1)$ satisfies the stated condition.
There is nothing suggesting that there is a maximal such value of $r$ after which no more unique strange roots may exist.
Based on the form of the condition in Prop~\ref{usr} we present the following conjecture.

\begin{conjecture}
There exists an infinite number of integers $r\in \N$ for which the size of the set $\{ n \in \N~:~ \sr(n)=r\}$ is $2$.
\end{conjecture}

Fagan's Question translates into the question that we have considered, 
since an integer $r$ is the strange root of only two integers iff 
$\{m \in \N ~:~\cf(m) = r\}$ is a singleton set.
In Figure~\ref{together} we record the first few values of both $\cf$ and $\sr$ to summarize how they are related.

\begin{figure}
$$
\begin{array}{|c|c|c|c|} \hline
m & \cf(m) & \sr(n) & n\\ \hline
1 & 2 & 2 & 2 \\
2 & 3 & 3 & 3,4\\
3 & 4 & 4 & 5,6\\
6 & 6 & 6 & 11,12\\
30 & 14 & 14 & 59,60 \\
493080 & 1760 & 1760 & 986159,986160\\
242650650 & 39046 & 39046 & 485301299,485301300 \\ \hline
\end{array}
$$\
\caption{Note that $\cf(m)=x$ is equivalent to $\sr(2m)=x$.  \label{together}}
\end{figure}

Since the numbers in Figure~\ref{together} seem to be growing so fast, it is not easy to get a clearer picture on the next value (if it exists). 
It would be interesting to see if some of the theory regarding the game of Tchoukaillon solitaire could be utilized 
to give insights into strange roots that are the strange roots of at most two integers.

We end this paper with a reformulation of Fagan's conjecture in terms of Tchoukaillon solitaire through the correspondence we established in Theorem~\ref{biject}.
Let us say that the Tchoukaillon solitaire board $b=(b_1,\ldots,b_k)=\Tchouk_n$ satisfies the {\it{Fagan property}} if 
\begin{enumerate}
\item[(i)] $b_1=1$ and $b_k=k$, and 
\item[(ii)] $1\leq b_i < i$ for all $i \in \{2,\ldots,k-1\}$.
\end{enumerate}
To see what the Fagan property represents from a solitaire perspective, let us assume that $\Tchouk_n$ satisfies the Fagan property.
Consider the next two plays of the $\Tchouk_n$ game.
We have, by Theorem~\ref{biject}(a), that $\sr(n)=1+k$.
Hole 1 contains 1 stone and should now be played since if it ends up containing more than one stone the game has been lost.
After playing hole 1 we will have $b_1=0$, $b_k=k$, and $1\leq b_i < i$ for all other $i$ which is the board $\Tchouk_{n-1}$.
By Theorem~\ref{biject}(a) we have $\sr(n-1)=1+k$.
We now play the leftmost hole such that $b_i=i$, and this is hole $k$. 
After this play we will have the configuration $(1,b_2+1,\ldots,b_{k-1}+1,0)=\Tchouk_{n-2}$, which is a board having length 1 less than $\Tchouk_{n-1}$.
By Theorem~\ref{biject}(a), we have $\sr(n-2)=k$ and this is different to $\sr(n-1)$ and $\sr(n)$.

Next consider $\Tchouk_{n+1}$, the board that results from `un-playing' the board $\Tchouk_n$. 
As all of the holes $1,\ldots,k$ of $\Tchouk_n$ are non-empty, it must be the case that the last hole played that resulted in $\Tchouk_n$ was hole $k+1$. 
This means $\sr(n+1) = 1+(k+1)$.
Therefore, if $b=(b_1,\ldots,b_k)=\Tchouk_{n}$ satisfies the Fagan property, there are at most two integers, $n$ and $n-1$, whose strange root is $k+1$.
\begin{conjecture}[Fagan's Tchoukaillon Solitaire Conjecture]
The exists an infinite number of Tchoukaillon solitaire boards that satisfy the Fagan property.
\end{conjecture}
\begin{remark}
Determining those Tchoukaillon boards that satisfy the Fagan property corresponds to solving a system of linear congruences (as explained in this paper just after Example~\ref{notany}), 
a fact that would suggest that a solution to this conjecture is difficult but not impossible.
\end{remark}

\section{Acknowledgments}
I am grateful to Colm Fagan for sharing his interesting problem with me and to the anonymous referee for very insightful comments that have helped improve both the presentation and content of the paper.
My thanks to Neil Sloane for providing a reference for the the construction due to Kevin Brown.

\end{document}